\documentclass[11pt]{article}
\usepackage{pgfplots}
\usepackage{amsmath}
\usepackage{amssymb}
\usepackage{algorithmic}
\usepackage{algorithm}
\usepackage{mathtools}
\usepackage{color}
\usepackage{tikz}
\usetikzlibrary{arrows,
	arrows.meta,
	bending,calc,patterns,angles,quotes}
\usepackage{float}

\newcommand{\rd}{{\rm d}}

\makeatletter
\providecommand*{\cupdot}{%
	\mathbin{%
		\mathpalette\@cupdot{}%
	}%
}
\newcommand*{\@cupdot}[2]{%
	\ooalign{%
		$\m@th#1\cup$\cr
		\hidewidth$\m@th#1\cdot$\hidewidth
	}%
}
\makeatother

\newtheorem{condition}{Condition}
\newtheorem{theorem}{Theorem}

\newtheorem{lemma}{Lemma}

\newtheorem{remark}{Remark}
\newtheorem{definition}{Definition}

\newtheorem{corollary}{Corollary}
\def\endpf{{\ \hfill\hbox{\vrule width1.0ex height1.0ex}\parfillskip 0pt
	}}
	\newenvironment{proof}{\noindent{\bf Proof:}}{\endpf}
\begin{document}
	\title{Simple sufficient condition for inadmissibility of Moran's single-split test}
	\author{Royi Jacobovic \thanks{Department of Statistics and Data-Science; The Hebrew University of Jerusalem; Jerusalem 9190501; Israel.
			{\tt royi.jacobovic@mail.huji.ac.il}}    \thanks{Department of Statistics; University of Haifa; Haifa 3498838; Israel.}}
	
	\date{\today}
	\maketitle
	\begin{abstract}
		 Suppose that a statistician observes two independent variates $X_1$ and $X_2$  having densities $f_i(\cdot;\theta)\equiv f_i(\cdot-\theta)\ ,\ i=1,2$ ,  $\theta\in\mathbb{R}$. His purpose is to conduct a test for
		 \begin{equation*}
		 H:\theta=0 \ \ \text{vs.}\ \ K:\theta\in\mathbb{R}\setminus\{0\}
		 \end{equation*}   
		 with a pre-defined significance level $\alpha\in(0,1)$.
		 Moran (1973) suggested a test which is based on  a single split of the data, \textit{i.e.,} to use $X_2$ in order to conduct a one-sided test in the direction of $X_1$. Specifically, if $b_1$ and $b_2$ are the $(1-\alpha)$'th and $\alpha$'th quantiles associated with the distribution of $X_2$ under $H$, then Moran's test has a rejection zone
		 \begin{equation*}
		 (a,\infty)\times(b_1,\infty)\cup(-\infty,a)\times(-\infty,b_2)
		 \end{equation*}
		 where  $a\in\mathbb{R}$ is a design parameter.
		 Motivated by this issue, the current work includes an analysis of a new notion, \textit{regular admissibility} of tests. It turns out that the theory regarding this kind of admissibility leads to a simple sufficient condition on $f_1(\cdot)$ and $f_2(\cdot)$ under which Moran's test is inadmissible. Furthermore, the same approach leads to a formal proof for the conjecture of DiCiccio (2018) addressing that the multi-dimensional version of Moran's test is inadmissible when the observations are $d$-dimensional Gaussians.
	\end{abstract}
	
	
	\section{Introduction}
	Suppose that a statistician observes a sample of $n\geq2$ independent variates each having the distribution $F(\cdot;\theta)$ ,   $\theta\in\Theta$. His purpose is to conduct a test for
	\begin{equation}\label{hypothesis: general}
	H:\theta=\theta_H \ \ \text{vs.}\ \ K:\theta\in\Theta\setminus\{\theta_H\}
	\end{equation}
	for some $\theta_H\in\Theta$. Once $\Theta$ is large, Moran \cite{Moran1973} suggested to use a test which is based on a single split of the data. Namely, divide the sample into two parts of sizes $1\leq m<n$ and $k=n-m$. The first part is used to estimate in what direction some estimator of
	$\theta$ diverges from $\theta_H$.  Then,  the second part is applied in order to test whether the true value of
	$\theta$ diverges from $\theta_H$ in this direction. The rational of this procedure is based on the expectation that the increase in power resulting from selecting a restricted alternative in the second stage would compensate for the reduction in power which is a consequence of the diminished sample size $k$.  
	
	For example, assume that  the observations are iid one-dimensional Gaussians with mean $\theta\in\mathbb{R}$ and unit variance. Then, Moran's test for $H:\theta=0$ versus $K:\theta\in\mathbb{R}\setminus\{0\}$ with a pre-defined significance level $\alpha\in(0,1)$ is as follows: If the mean of the first sub-sample is positive, then use the mean of the second sub-sample to conduct an $\alpha$-level upper-tailed Z-test. Otherwise, use  the mean of the second sub-sample to conduct an $\alpha$-level lower-tailed Z-test. Intuitively, it seems reasonable to expect that once $\theta$ is far away from the origin, then Moran's test leads to some gain relative to an $\alpha$-level two-sided Z-test. Interestingly, this intuition fails because Moran's suggestion is an unbiased test and it is well-known that the two-sided Z-test is a uniformly most powerful (UMP) unbiased test. In fact, this indicates that for this particular case, Moran's test is either inadmissible or a UMP unbiased test. 
	
	The current research query is whether Moran's test is inadmissible under general assumptions regarding the data distribution? Specifically, assume that the mean of the $i$'th ($i=1,2$) sub-sample has a distribution with a density function $f_i(\cdot;\theta)\equiv f_i(\cdot-\theta)$ , $\theta\in\mathbb{R}$. In this setup,  let $b_1$ and $b_2$ be the $(1-\alpha)$'th and $\alpha$'th quantiles of the distribution whose density is $f_2(\cdot;0)$. Then, Moran's test for $H:\theta=0$ versus $K:\theta\in\mathbb{R}\setminus\{0\}$ has a rejection zone
	\begin{equation}
	(a,\infty)\times(b_1,\infty)\cup(-\infty,a)\times(-\infty,b_2)
	\end{equation}
	where $a\in\mathbb{R}$ is a design parameter. In this work, it is shown that the following condition is sufficient for inadmissibility of Moran's test.
\newpage	
	\begin{condition}\label{cond: inadmissibility}
		$\text{ }$\begin{description}
			\item[(a)] For every $i=1,2$ , $f_i(\cdot)$ is continuous, bounded and positive on $\mathbb{R}$.
			
			\item [(b)] For every $\theta\in(0,\infty)$:
			\begin{equation}
			\exists\lim_{x\to\infty}\frac{f_{2}(x-\theta)}{f_{2}(x)}=\infty\ \ , \ \ \exists\lim_{x\to-\infty}\frac{f_{2}(x-\theta)}{f_{2}(x)}<\infty\,.
			\end{equation}
			
			\item[(c)]For every $\theta\in(-\infty,0)$:
			\begin{equation}
			\exists\lim_{x\to\infty}\frac{f_{2}(x-\theta)}{f_{2}(x)}<\infty\ \ , \ \ \exists\lim_{x\to-\infty}\frac{f_{2}(x-\theta)}{f_{2}(x)}=\infty\,.
			\end{equation}
		\end{description}
	\end{condition}
	In particular, when $f_2(\cdot)$ is even, \textbf{(b)} and \textbf{(c)} may be unified into the condition
	\begin{equation}
	\exists\lim_{x\to\infty}\frac{f_{U_2}(x-\theta)}{f_{U_2}(x)}=\infty\ \ , \ \ \forall\theta\in(0,\infty)\,.
	\end{equation}
	Then, it is easy to apply this condition in order to prove that Moran's test is inadmissible when the data has a Gaussian distribution.
	
	Allegedly, it seems promising to prove  Condition \ref{cond: inadmissibility} by specifying another test and show directly that its power uniformly dominates the power of Moran's test. In practice, as mentioned by Moran \cite{Moran1973}, this approach is not applicable even when the relatively simple Gaussian setup is considered. Another methodology is to pinpoint another test which is known to be UMP in certain class of tests (\textit{e.g.,} the class of unbiased tests). Then, it is left to show that: (1) Moran's test belongs to this class. (2) Moran's test is not UMP in this class. Provided the theory regarding UMP unbiased tests (see, \textit{e.g.,} Sections 4 and 5 in \cite{Lehmann2006}), this approach sounds reasonable only for specific special cases like the Gaussian one. However, it is not clear how this approach should work out when the general case is under consideration. 
	
	An additional contribution of this work is by showing that Moran's test for the multi-dimensional Gaussian setup is inadmissible. More precisely, assume that the observations are  $d$-dimensional Gaussians ($2\leq d<\infty$) with mean $\theta\in\mathbb{R}^d$ and an identity covariance matrix. Then, the purpose of the statistician is to test
	\begin{equation}
	H:\theta=\textbf{0} \ \ \text{vs.} \ \ K:\theta\in\mathbb{R}^d\setminus\{\textbf{0}\}
	\end{equation}
	where $\textbf{0}$ is the zero-vector in $\mathbb{R}^d$. Let $\bar{X}_1$ be the mean of the first sub-sample. Then, Moran \cite{Moran1973} suggested to use the second sub-sample in order to test a simple hypothesis in the direction of $\bar{X}_1$. That is, consider a most powerful (MP) test for the simple hypothesis testing
	\begin{equation}
	H:\theta=\textbf{0} \ \ , \ \ K(\bar{X}_1):\theta=\frac{\bar{X}_1}{\|\bar{X}_1\|}
	\end{equation}
	with a pre-defined significant level $\alpha\in\left(0,1\right)$ where $\|\cdot\|$ is the Euclidean norm in $\mathbb{R}^d$. Standard likelihood-ratio calculations yield that the rejection zone of this test equals to
	\begin{equation}
	\left\{x_1,x_2,\ldots,x_n\in\mathbb{R}^d\ \text{s.t.}\ \bar{x}_1\neq\textbf{0}\ ; \ \frac{\bar{x}_1\cdot\bar{x}_2}{\|\bar{x}_1\|}>D\right\}
	\end{equation}
	where $\cdot$ denotes dot product, $\bar{x}_1\equiv\frac{1}{m}\sum_{i=1}^{m}x_i$, $\bar{x}_2\equiv\frac{1}{k}\sum_{i=m+1}^nx_i$ and $D>0$ is a constant which is determined uniquely by the vector $(\alpha,m,n)$. 
	
	At a first glance, one may suspect that the inadmissibility of Moran's test stems directly from invariance considerations, but some cautions are needed. Let $G$ be the group of orthogonal transformations on $\mathbb{R}^d$ and observe that Moran's test is invariant with respect to $G$. In addition, the (whole) sample mean is a sufficient statistic for the data distribution. Thus, by the so called `sufficiency principle', it makes sense to focus on tests which are determined uniquely by the sample mean. It is well-known that in this class of tests, the UMP $G$-invariant test is the celebrated chi-square test. Thus, it is tempting to believe that the chi-square test is also UMP $G$-invariant test in the class of all tests.  A problem follows since this deduction is by all means non-trivial (see, \textit{e.g.,} the last paragraph before Example 6.3.3 in \cite{Lehmann2006} along with the references therein). Moreover, even if it were possible to justify this deduction for this particular case, it would be left to show that there exists $\theta\in\mathbb{R}\setminus\{\textbf{0}\}$ for which the chi-square test is better than Moran's test. Consequently, inadmissibility of Moran's test is not simply originated from the existing literature. In particular, this fact motivated DiCiccio \cite{DiCiccio2018} to develop an approximation (as $d\to\infty$) of the power function of Moran's test. This approximation was used in order to show that once $d$ is large, the chi-square test outperforms Moran's test (see Remark 2.2 in \cite{DiCiccio2018}). In addition, \cite{DiCiccio2018} also includes a simulation study in support of this claim.

	Motivated by these applications, the theoretical contribution of this work is due to an analysis of a new notion, \textit{regular admissibility} of tests. Section \ref{sec: preliminaries} includes a detailed description of a general model in which the theory regarding regular admissibility is established. Roughly speaking, the main result of Section \ref{sec: regular admissibility} is that any test which satisfies regular admissibility is an MP test in a simple hypothesis testing problem of the original null $H$ versus a Bayes-mixture of simple alternatives under $K$. In Section \ref{sec: inadmissibility}, it is shown how the fundamental Lemma of Neyman and Pearson naturally leads to a sufficient condition for inadmissibility. This condition is applied in order to show that Moran's test is inadmissible. Specifically, Section \ref{sec: univariate} includes an application of the results from the previous sections to the proof of Condition \ref{cond: inadmissibility}.  
	Section \ref{sec: multivariate} is about another application to the multi-dimensional Gaussian case. Finally, Section \ref{sec:conclusion} is a short discussion regarding the way that the current results should be interpreted including several directions for further research. 
	 
	\subsection{Related literature}
	In \cite{Moran1973}, Moran wrote that data-splitting is an approach in which much of natural
	science proceeds, \textit{i.e.,}  one examines a large body of evidence, notices certain empirical features, and then proceeds to test if these are real. Despite this fact, as mentioned by Romano \textit{et al.} \cite{Romanov2019}, it turns out that the statistical literature does not contain much work in this direction. Besides \cite{Moran1973}, there is another work of Cox \cite{Cox1975,Cox1977} which is about an application of data-splitting to an hypothesis testing regarding the mean of a Gaussian population with a sparse alternative.
	
	Naturally, tests which are based on data-splitting are exposed to criticism because they are  not invariant to permutations of the data. In order to relieve this effect, much of the modern research in this area is focused on tests which are based on many splits (unlike \cite{Cox1975,Cox1977} and \cite{Moran1973} which are based on a single split). Namely, from each sub-sample it is possible to extract a p-value. Then, the question has become how to combine these p-values? For several works in this direction see, \textit{e.g.,} \cite{DiCiccio2018,DiCiccio2020,Romanov2019} and the references therein.
	
	Recently, Vovk \textit{et al.} \cite{Vovk2019,Vovk2020} suggested to use e-values instead of p-values. By doing so, he claims that the problem of using a test which is not invariant to permutations of the data becomes much less serious. For more information about e-values, see \cite{Shafer2020}. 

	Another branch of literature which is related to the current work is about complete class theorems for hypothesis testing problems with simple null versus a composite alternative. Several works in this direction are \textit{e.g.,} \cite{Brown1989,Brown1992,Marden1982}. In addition, some relevant surveys regarding admissibility are \cite{Johnstone2019,Rukhin1995}. Importantly, note that there is a literature about admissibility and complete class theorems in the broader context of decision theory. An introduction to this topic is given in Section 2 of \cite{Ferguson2014}. For a more advanced treatment, see, \textit{e.g.,} \cite{Asgharian1998,Kozek} and the references therein. 
	
	\section{Preliminaries} \label{sec: preliminaries}
	Consider a measurable space $\left(\Omega,\mathcal{H}\right)$ and a finite-dimensonal Banach space $(\Theta,\|\cdot\|)$. In particular, let $\textbf{0}$ be the zero vector which is associated with $\Theta$. Then,  it is known that both $\Theta$ and $\Theta_K\equiv\Theta\setminus\{\textbf{0}\}$ are locally compact and $\sigma$-compact. In addition, assume that $X$ is a $\left(\Omega,\mathcal{H}\right)$-measurable function  which receives values in some metric space $\mathcal{X}$. Then, for every $\theta\in\Theta$, let $P_\theta$ be a probability measure on $\left(\mathcal{X},\mathcal{B}(\mathcal{X})\right)$ where $\mathcal{B}$ is the Bor\'el $\sigma$-field of the corresponding topological space. 
	Then, assume that a statistician who does not know the value of $\theta$ observes $X$ which is a realization of $P_\theta$. His intention is to test 
	\begin{equation}\label{eq: hypothesis statement}
	H:\theta=\textbf{0} \ \ \text{vs.} \ \ K:\theta\in \Theta_K
	\end{equation}
	with a pre-defined significance level $\alpha\in(0,1)$. 
	
	Denote 
	\begin{equation}
	\Phi\equiv\left\{\phi:\mathcal{X}\rightarrow[0,1]\ ;\ \phi\text{ is a Bor\'el function}\right\}
	\end{equation}
	which is the set of all  tests and for every $\phi\in\Phi$ define the power function
	\begin{equation}
	\beta_\phi(\theta)\equiv\int_\Theta\phi(x)\rd P_\theta(x)\ \ , \ \ \forall\theta\in\Theta
	\end{equation} 
	and its complement $\gamma_\phi\equiv1-\beta_\phi$. In particular, note that for every $\theta\in\Theta_K$, $\gamma_\phi(\theta)$ equals to the risk which is associated with the test $\phi$ under $\theta$.
	In addition, notice that if $t\in(0,1)$ and $\phi_1,\phi_2\in\Phi$, then 
	\begin{equation}
	\phi_t\equiv t\phi_1+(1-t)\phi_2\in\Phi
	\end{equation}
	and 
	\begin{equation}
	\beta_{\phi_t}(\theta)=t\beta_{\phi_1}(\theta)+(1-t)\beta_{\phi_2}(\theta)\ \ , \ \ \forall\theta\in\Theta\,.
	\end{equation}

	Now, for every  $\Gamma\in\left\{\Theta,\Theta_K\right\}$ let $\mathcal{C}(\Gamma)$ be the set of all functions $f:\Gamma\rightarrow\mathbb{R}$ which are continuous on $\Gamma$. Then, denote the space of continuous functions which vanish in infinity by $\mathcal{C}_0(\Gamma)$. Namely, $f\in\mathcal{C}_0(\Gamma)$ if and only if (iff) $f\in\mathcal{C}(\Gamma)$ and for every $\epsilon>0$, there exists a compact set $K_f(\epsilon)\subseteq \Gamma$ such that
	\begin{equation}
	\{\theta\in \Gamma\ ;\ |f(\theta)|\geq\epsilon\}\subseteq K_f(\epsilon)\,.
	\end{equation}
	Similarly, $\mathcal{C}_c(\Gamma)$ is a notation for the space of all compactly supported functions. Namely, $f\in \mathcal{C}_c(\Gamma)$ iff $f\in\mathcal{C}(\Gamma)$ and there is a compact set $K_f\subseteq \Gamma$ such that
	\begin{equation}
	\{\theta\in \Gamma\ ;\ f(\theta)\neq0\}\subseteq K_f\,.
	\end{equation}
	The following simple lemma will be helpful later on. In order to make its statement, for every function $f:\Theta\rightarrow\mathbb{R}$, let $f\big|_{\Theta_K}$ be the restriction of $f$ to $\Theta_K$.
	
	\begin{lemma}\label{lemma: contradicting function}
		 There exists a  function $f\in\mathcal{C}_0(\Theta)$ such that:
		\begin{enumerate}
			\item $f\big|_{\Theta_K}\in\mathcal{C}_0(\Theta_K)$.
			
			\item $f(\textbf{0})<f(\theta)$ for every $\theta\in\Theta_K$. 
		\end{enumerate}   
	\end{lemma}
	
	\begin{proof}
		Consider a function
		
		\begin{equation}
		f(\theta)\equiv\begin{cases} 
		\|\theta\| & \|\theta\|\leq 1 \\
		e^{-\left(\|\theta\|-1\right)} & \|\theta\|>1 
		\end{cases}\ \ , \ \ \forall\theta\in\Theta\,.
		\end{equation}
		It is given that $\Theta$ is a finite-dimensional and hence it is enough to check that for every $\epsilon>0$, the set $\{\theta\in\Theta;f(\theta)\geq\epsilon\}$ is closed, bounded and contained in $\Theta_K$. This as well as the other requirements of the lemma are easy to verify.
	\end{proof}
	
	\section{Regular admissibility and $\alpha$-level Bayes tests}\label{sec: regular admissibility}
	Consider the following definitions: 
	\begin{definition}\label{def: regular test}
		Let 
		\begin{equation}
		\Phi(\alpha)\equiv\left\{\phi\in\Phi\ ;\ \beta_\phi(\textbf{0})=\alpha \right\}\,.
		\end{equation}
		Then, $\phi\in\Phi(\alpha)$ is a \textit{regular} test (of level $\alpha$) iff $\gamma_\phi\in \mathcal{C}_0(\Theta)$. In addition, let $\Phi_0(\alpha)$ be the set of all regular tests (of level $\alpha$). 
	\end{definition}

	\begin{definition}\label{def: admissible test}
		Let $\phi\in S\subseteq\Phi(\alpha)$. Then, $\phi$ is an admissible test in $S$ iff for every $\phi'\in S$, the condition 
		\begin{equation}
		\beta_\phi(\theta)\leq\beta_{\phi'}(\theta)\ \ , \ \ \forall\theta\in\Theta_K
		\end{equation}
		implies that 
		\begin{equation}
		\beta_\phi(\theta)=\beta_{\phi'}(\theta)\ \ , \ \ \forall\theta\in\Theta_K\,.
		\end{equation}
		Otherwise, $\phi$ is inadmissible in $S$. Correspondingly, denote the set of all admissible tests in $S$ by $\mathcal{A}_S(\alpha)$. 
		
		In particular, denote the set  $\mathcal{A}_0(\alpha)\equiv\mathcal{A}_{\Phi_0(\alpha)}(\alpha)$ and each test $\phi\in\mathcal{A}_0(\alpha)$ is said to have the property of regular admissibility. 
	\end{definition}
	
	\begin{definition}\label{def: bayes test}
	Let $\Pi$ be the set of all probability measures on $\left(\Theta_K,\mathcal{B}(\Theta_K)\right)$ and consider some $\pi\in\Pi$. In addition, let $S\subseteq\Phi(\alpha)$ be such that for every $\phi'\in S$, $\theta\mapsto\gamma_{\phi'}(\theta)$ is a Bor\'el function on $\Theta_K$. Then, $\phi\in S$ is an $\alpha$-level $\pi$-Bayes test in $S$ iff 
	\begin{equation}\label{eq: bayes definition}
	\int_{\Theta_K}\beta_{\phi}(\theta)\rd\pi(\theta)\geq\int_{\Theta_K}\beta_{\phi'}(\theta)\rd \pi(\theta)\ \ , \ \ \forall\phi'\in S\,.
	\end{equation}

	\end{definition}
	
	\begin{remark}\normalfont
		To the best of the author's knowledge, the concept of a regular test as described in Definition \ref{def: regular test} is new. Generally speaking, this form of regularity is consistent with the intuition that the power of a test should tend to one as the alternative diverges from the null. From a technical perspective,  unlike in other papers (see, \textit{e.g.,} \cite{Brown1989,Brown1992}), power functions of regular tests are not required to satisfy any differentiability condition.   
	\end{remark}

	\begin{remark}\normalfont
		Regular admissibility is a concept which is derived from the standard concept of  admissibility  (see, \textit{e.g.,} Equation (6.18) in \cite{Lehmann2006}) by ignoring all tests which are not regular in the sense of Definition \ref{def: regular test}. For other concepts of admissibility which are related to hypothesis testing, see, \textit{e.g.,} Section 6.7 in \cite{Lehmann2006} and \cite{Brown1992} with the references therein.  
	\end{remark}
	
	\begin{remark}
		\normalfont Albeit the notion which is introduced in Definition \ref{def: bayes test} resembles the standard definition of a Bayes rule (see, \textit{e.g.,} Equation (1.15) in \cite{Ferguson2014}), there is a difference. To see this, take $\pi$ which is a probability measure on $\left(\Theta,\mathcal{B}(\Theta)\right)$ such that $\pi(\Theta_K)=1$. In addition, denote the restriction of $\pi$ to $\Theta_K$ by $\pi\big|_{\Theta_K}$. It is straightforward that a statistician with prior belief which is represented by $\pi$ will reject the null hypothesis for sure. On the other hand, if $\phi$ is an  $\alpha$-level $\pi\big|_{\Theta_K}$-Bayes test, then $\beta_\phi(\textbf{0})=\alpha<1$ which makes it different.  
	\end{remark}

	\begin{theorem}\label{thm: sufficient condition}
		If $\phi\in\mathcal{A}_0(\alpha)$, then there exists $\pi\in\Pi$ for which $\phi$ is an $\alpha$-level $\pi$-Bayes test in $\Phi_0(\alpha)$.  
	\end{theorem}
	\begin{proof}
		Denote 
		\begin{equation}
		\Delta(\phi)\equiv\left\{f\in \mathcal{C}(\Theta_K)\ ;\ f(\theta)\leq\gamma_\phi(\theta) \ ,  \ \forall\theta\in\Theta_K\right\}
		\end{equation}
		and observe that 
		\begin{equation}
		\mathcal{S}\equiv\left\{\gamma_{\phi'}\big|_{\Theta_K}\ ;\ \phi'\in\Phi_0(\alpha)\right\}\subseteq\mathcal{C}(\Theta_K)\,.
		\end{equation}
		 It is given that $\phi\in\mathcal{A}_0(\alpha)$ and hence 
		\begin{equation}
		\left(\Delta(\phi)\setminus\{\gamma_\phi\big|_{\Theta_K}\}\right)\cap\mathcal{S}=\emptyset\,.
		\end{equation} 
		In fact, it can be verified that both $\mathcal{S}$ and $\Delta(\phi)\setminus\{\gamma_\phi\big|_{\Theta_K}\}$ are convex sets. Especially, since $\gamma_\phi(\theta)\geq 0$ for every $\theta\in\Theta_K$, then 
		\begin{equation}
		f(\theta)=-1\ \ , \ \ \forall\theta\in\Theta_K
		\end{equation}
		is an interior point of $\Delta(\phi)\setminus\{\gamma_\phi\big|_{\Theta_K}\}$ with respect to the norm
		\begin{equation}
		\|f\|_{\infty|\Theta_K}\equiv\sup_{\theta\in\Theta_K}|f(\theta)|\ \ , \ \ \forall f\in\mathcal{C}(\Theta_K)\,.
		\end{equation}
		Therefore, the basic separation theorem (see Theorem 3.5.13 and Corollary 3.5.14(a) in \cite{Ash2014}) implies that there exists a non-zero continuous linear functional $\mathcal{F}\{\cdot\}$ on $\mathcal{C}(\Theta_K)$ such that
		
		\begin{equation}\label{eq: supporting hyperplane}
		\mathcal{F}\{f\}\leq\mathcal{F}\{\gamma\}\ \ , \ \ \forall f\in\Delta(\phi)\setminus\{\gamma_\phi\big|_{\Theta_K}\}\ , \ \gamma\in\mathcal{S}\,.
		\end{equation} 
		Now, consider some $f\in\mathcal{C}(\Theta_K)$ which is nonpositive and not identically zero. Then,  $f+\gamma_\phi\big|_{\Theta_K}\in\Delta(\phi)\setminus\{\gamma_\phi\big|_{\Theta_K}\}$ and hence
		\begin{equation}
		\mathcal{F}\{f\}+\mathcal{F}\{\gamma_\phi\big|_{\Theta_K}\}=\mathcal{F}\{f+\gamma_\phi\big|_{\Theta_K}\}\leq\mathcal{F}\{\gamma_\phi\big|_{\Theta_K}\}\,.
		\end{equation}
		Moreover, $\mathcal{F}\{\cdot\}$ is continuous and hence it is also bounded, \textit{i.e.,} there exists a constant $M\in(0,\infty)$ such that
		\begin{equation}\label{eq: bounded functional}
		\big|\mathcal{F}\{f\}\big|\leq M\|f\|_{\infty|\Theta_K}\ \ , \ \ \forall f\in \mathcal{C}(\Theta_K)\,.
		\end{equation} 
		Therefore, since $\|\gamma_\phi\big|_{\Theta_K}\|_{\infty|\Theta_K}\leq1$, deduce that $|\mathcal{F}\{\gamma_\phi\big|_{\Theta_K}\}|<\infty$ and hence $\mathcal{F}\{\cdot\}$ is positive, \textit{i.e.,}
		\begin{equation}
		\mathcal{F}\{f\}\geq0\ \ , \ \ \forall f\in\mathcal{C}(\Theta_K)\ \  \text{s.t.}\ \  f(\theta)\geq0\ , \ \forall \theta\in\Theta_K\,.
		\end{equation}
		$(\Theta_K,\|\cdot\|)$ is a special case of locally compact $\sigma$-compact Hausdorff space. 
		Thus, Riesz-Markov-Kakutani representation theorem (see Section 2.14  in \cite{Rudin}) states that there exists a  regular Bor\'el measure $\mu$ on $\Theta_K$ such that
		
		\begin{equation}\label{eq: riesz representation1}
		\mathcal{F}\{f\}=\int_{\Theta_K}f(\theta)\rd\mu(\theta)\ \ , \ \ \forall f\in\mathcal{C}_c(\Theta_K)\,.
		\end{equation}
		Since $\mathcal{F}\{\cdot\}$ is non-zero functional, then $\mu$ is a non-zero measure. In addition, 
		since $\Theta_K$ is $\sigma$-compact, then there exists a sequence of non-empty compact sets $K_1\subseteq K_2\subseteq\ldots\subseteq\Theta_K$ such that $\cup_{n=1}^\infty K_n=\Theta_K$.
		Therefore, since $\Theta_K$ is an open set, then  Urysohn's lemma (see Section 2.12 in \cite{Rudin}) implies that for every $n\geq1$ there exists a function $g_n\in \mathcal{C}_c(\Theta_K)$ such that 
		\begin{enumerate}
			\item $0\leq g_n(\theta)\leq1$ for every $\theta\in\Theta_K$.
			
			\item $g_n(\theta)=1$ for every $\theta\in K_n$.
		\end{enumerate}
		Thus, observe that for every $n\geq1$
		\begin{align}
		\mu\left(K_n\right)&=\int_{K_n}\rd\mu(\theta)\\&\leq\int_{\Theta}g_n(\theta)\rd\mu(\theta)=\mathcal{F}\{g_n\}\leq M\|g_n\|_{\infty|\Theta_K}=M<\infty\,.\nonumber
		\end{align}
		This means that by taking $n\to\infty$, continuity of measure (from below) implies that $\mu\left(\Theta_K\right)$ is finite.  
		
		Let $f$ be some arbitrary function in $\mathcal{C}_0(\Theta_K)$. It is well-known that $\mathcal{C}_c(\Theta_K)$ is dense in $\mathcal{C}_0(\Theta_K)$ with respect to the norm $\|\cdot\|_{\infty|\Theta_K}$. This means that there exists a sequence $(h_n)_{n\geq1}\subseteq\mathcal{C}_c(\Theta_K)$ such that $\|f-h_n\|_{\infty|\Theta_K}\rightarrow0$ as $n\to\infty$. Therefore, since $\mathcal{F}\{\cdot\}$ is continuous, then $\mathcal{F}\{h_n\}\rightarrow\mathcal{F}\{f\}$ as $n\to\infty$. In addition, $\|f\|_{\infty|\Theta_K}<\infty$ and hence 
		\begin{equation}
		\sup_{n\geq1}\|h_n\|_{\infty|\Theta_K}=\sup_{n\geq1}\sup_{\theta\in\Theta_K}|h_n(\theta)|<\infty\,.
		\end{equation} 
		At the same time, the above-mentioned uniform convergence implies pointwise convergence $h_n\rightarrow f$ as $n\to\infty$ on $\Theta_K$. Thus, recalling that $\mu$ is a finite measure, then bounded convergence theorem yields that
		\begin{equation}
		\mathcal{F}\{h_n\}=\int_{\Theta_K} h_n(\theta)\rd\mu(\theta)\xrightarrow{n\to\infty}\int_{\Theta_K} f(\theta)\rd\mu(\theta)\,.
		\end{equation}   Hence, by the uniqueness of the limit and the generality of $f$, deduce that
		\begin{equation}\label{eq: representation1}
		\mathcal{F}\{f\}=\int_{\Theta_K} f(\theta)\rd\mu(\theta)\ \ , \ \ \forall f\in\mathcal{C}_0(\Theta_K)\,.
		\end{equation}
		Now, define a new non-zero linear functional
		\begin{equation}
		\mathcal{G}\{f\}\equiv\mathcal{F}\left\{f\big|_{\Theta_K}\right\}
		\end{equation}
		on a normed vector-space $\left(\mathcal{C}(\Theta),\|\cdot\|_{\infty|\Theta}\right)$
		where
		\begin{equation}
		\|f\|_{\infty|\Theta}\equiv\sup_{\theta\in\Theta}|f(\theta)|\ \ , \ \ \forall f\in\mathcal{C}(\Theta)\,.
		\end{equation}
		Observe that for every $f\in\mathcal{C}(\Theta)$,
		\begin{equation}
		|\mathcal{G}\{f\}|=|\mathcal{F}\{f\big|_{\Theta_K}\}|\leq M\sup_{\theta\in\Theta_K}|f(\theta)|\leq M\|f\|_{\infty|\Theta}
		\end{equation} 
		which means that $\mathcal{G}\{\cdot\}$ is bounded and hence also continuous.
		Thus, as it was shown for $\mathcal{F}\{\cdot\}$, it is possible to show that $\mathcal{G}\{\cdot\}$ is  positive. Consequently,  since $(\Theta,\|\cdot\|)$ is a special case of a locally compact $\sigma$-compact Hausdorff space, Riesz-Markov-Kakutani  representation theorem can be applied once again. This time in order to show  that there exists a regular Bor\'el measure $\nu$ on $\Theta$ such that
		\begin{equation}\label{eq: riesz representation}
		\mathcal{G}\{f\}=\int_{\Theta}f(\theta)\rd\nu(\theta)\ \ , \ \ \forall f\in\mathcal{C}_c(\Theta)\,.
		\end{equation}
		In particular, $\nu$ is non-zero and finite by analogue arguments to those which were used in order to justify these properties of $\mu$. Furthermore, analogue arguments to those which lead to \eqref{eq: representation1}, imply that  
		\begin{equation}\label{eq: representation}
		\mathcal{G}\{f\}=\int_\Theta f(\theta)\rd\nu(\theta)\ \ , \ \ \forall f\in\mathcal{C}_0(\Theta)\,.
		\end{equation}	
		Now, assume by contradiction that $\nu$ is a Dirac measure on $\{\textbf{0}\}$. Notice that Lemma \ref{lemma: contradicting function} implies that there exists a bounded function $f\in\mathcal{C}_0(\Theta)$ such that:
		\begin{enumerate}
			\item $f\big|_{\Theta_K}\in\mathcal{C}_0(\Theta_K)$.
			
			\item  $f(\textbf{0})<f(\theta)$ for every $\theta\in\Theta_K$.
		\end{enumerate}  
		Then, a contradiction follows from
		\begin{align}
		f(\textbf{0})&=\mathcal{G}\{f\}=\mathcal{F}\{f\big|_{\Theta_K}\}=\int_{\Theta_K} f\big|_{\Theta_K}(\theta)\rd\mu(\theta)>f(\textbf{0})\,.\nonumber
		\end{align}	
		Fix some $\phi'\in\Phi_0(\alpha)$ and for every $n\geq1$, let
		\begin{equation}
		f_n(\theta)\equiv\gamma_\phi(\theta)-\frac{1}{n}\ \ , \ \ \forall\theta\in\Theta\,.
		\end{equation} 
		Clearly, for every $n\geq1$, $f_n\big|_{\Theta_K}\in\Delta(\phi)\setminus\{\gamma_\phi\}\subseteq\mathcal{C}(\Theta_K)$ and observe that
		\begin{equation}
		\|\gamma_\phi\big|_{\Theta_K}-f_n\big|_{\Theta_K}\|_{\infty|\Theta_K}\rightarrow0\ \ \text{as} \ \ n\to\infty\,.
		\end{equation}  Hence, since $\mathcal{F}\{\cdot\}$ is a continuous functional on $\mathcal{C}(\Theta_K)$, then $\mathcal{F}\{f_n\big|_{\Theta_K}\}\rightarrow\mathcal{F}\{\gamma_\phi\big|_{\Theta_K}\}$ as $n\to\infty$. Due to \eqref{eq: supporting hyperplane}, deduce that $\mathcal{F}\{f_n\big|_{\Theta_K}\}\leq\mathcal{F}\{\gamma_{\phi'}\big|_{\Theta_K}\}$ for every $n\geq1$. Thus, taking the limit $n\to\infty$ yields that $\mathcal{F}\{\gamma_\phi\big|_{\Theta_K}\}\leq\mathcal{F}\{\gamma_{\phi'}\big|_{\Theta_K}\}$ and hence
		\begin{align}
		\int_\Theta\gamma_\phi(\theta)\rd\nu(\theta)&=\mathcal{G}\{\gamma_\phi\}=\mathcal{F}\{\gamma_\phi\big|_{\Theta_K}\}\\&\leq\mathcal{F}\{\gamma_{\phi'}\big|_{\Theta_K}\}=\mathcal{G}\{\gamma_{\phi'}\}=\int_{\Theta}\gamma_\phi(\theta)\rd\nu(\theta)\,.\nonumber
		\end{align} 
		Consequently, the generality of $\phi'$ implies that 
		\begin{equation}
		\int_\Theta\beta_\phi(\theta)\rd\nu(\theta)\geq\int_\Theta\beta_{\phi'}(\theta)\rd\nu(\theta)\ \ , \ \ \forall\phi'\in\Phi_0(\alpha)\,.
		\end{equation}
		Finally, since $\nu$ is a non-zero finite regular Bor\'el measure on $\Theta$ which is not concentrated on $\{\textbf{0}\}$ and $\beta_\phi(\textbf{0})=\beta_{\phi'}(\textbf{0})=\alpha$, then the result follows by setting $\pi$ to be the restriction of $\nu$ to $\Theta_K$ with a proper normalization.
	\end{proof}
	
	\begin{remark}
		\normalfont In fact, the measure $\pi$ which is defined in the statement of Theorem \ref{thm: sufficient condition} is regular. That is, $\pi$ satisfies the regularity conditions which appear in the statement of Riesz-Markov-Kakutani representation theorem (see Section 2.14  in \cite{Rudin}). Since the regularity of $\pi$ is not important for the  analysis to follow in the next sections, this fact was omitted from the statement of Theorem \ref{thm: sufficient condition}.
	\end{remark}
	\begin{remark}\normalfont
		Regarding the proof of Theorem \ref{thm: sufficient condition}, one may be wondering why not to take $\pi$ to be $\mu$? To answer this question, for simplicity assume that $\Theta=\mathbb{R}$ which means that $\Theta_K=\mathbb{R}\setminus\{0\}$. In addition, consider a test $\phi\in\Phi_0(\alpha)$ such that $\theta\mapsto\gamma_\phi(\theta)$ is increasing   on $(-\infty,0)$ and decreasing on $(0,\infty)$. In such a case, for every $0<\epsilon<1-\alpha$, the set $\{\theta\in\Theta_K;\gamma_\phi(\theta)\geq\epsilon\}$ is not closed and hence not compact. This means that $\gamma_\phi\notin\mathcal{C}_0(\Theta_K)$ and consequently the proof of Theorem \ref{thm: sufficient condition} would be incorrect if one took $\pi$ which equals to $\mu$.   
	\end{remark}

	\section{Inadmissibility of regular tests}\label{sec: inadmissibility}
	Assume that $(B,\theta)\mapsto P_\theta(B)$ is a  transition kernel on $\Theta_K\times\mathcal{B}(\mathcal{X})$. Specifically, this means that for every $B\in\mathcal{B}(\mathcal{X})$, $\theta\mapsto P_\theta(B)$ is a Bor\'el function on $\Theta_K$ and for every $\theta\in\Theta$, $P_\theta$ is a probability measure on $\left(\mathcal{X},\mathcal{B}\left(\mathcal{X}\right)\right)$. In addition, assume that there exists a $\sigma$-finite measure $\lambda$ on $\mathcal{X}$ such that 
	\begin{equation}\label{eq: absolutely continuous}
	P_\theta\ll\lambda \ \ , \ \ \forall \theta\in\Theta
	\end{equation}
	where $\ll$ is a symbol for absolute continuity. Correspondingly, for every $\theta\in\Theta$  let $f_\theta$ be the nonnegative version of the Radon-Nikodym derivative of $P_\theta$ with respect to $\lambda$. Then, for every probability measure $\pi\in\Pi$,  define a probability measure $\mathcal{P}_\pi$ on $\left(\mathcal{X},\mathcal{B}(\mathcal{X})\right)$ such that
	
	\begin{equation}
	\mathcal{P}_\pi(B)=\int_\Theta P_\theta(B)\rd\pi(\theta)\ \ , \ \ \forall B\in\mathcal{B}(\mathcal{X})\,.
	\end{equation}
	In particular,  $\mathcal{P}_\pi$ may be considered as a Bayes mixture of simple alternatives under $K$ (this terminology is taken from Birnbaum \cite{Birnbaum1954})
	and \eqref{eq: absolutely continuous} implies that $\mathcal{P}_\pi\ll\lambda$. 
	
	Definition \ref{def: bayes test} actually asserts that an $\alpha$-level $\pi$-Bayes test in $\Phi(\alpha)$ is an MP test of the simple hypothesis testing problem: 
	\begin{equation}\label{eq: Bayes hypothesis statement}
	H:X\sim P_{\textbf{0}} \ \ \text{vs.} \ \ K_\pi:X\sim \mathcal{P}_\pi\,.
	\end{equation}
	Therefore, the fundamental lemma of Neyman and Pearson (see Theorem 3.2.1(iii) in \cite{Lehmann2006}) might be carried out under the assumption that there is no trivial test for \eqref{eq: Bayes hypothesis statement}. Specifically, if $\phi\in\Phi(\alpha)$ is an $\alpha$-level $\pi$-Bayes test in $\Phi(\alpha)$, then there exists a constant $C\equiv C_\pi\in\mathbb{R}$ such that 
	\begin{equation}\label{eq:N-P necessary condition}
	\phi(x)=\begin{dcases}
	1 & Cf_{\textbf{0}}(x)<\int_{\Theta_K}f_\theta(x)\rd\pi(\theta) \\
	0 & Cf_{\textbf{0}}(x)>\int_{\Theta_K}f_\theta(x)\rd\pi(\theta)
	\end{dcases}\ \ \ , \ \ \ \lambda\text{-a.s.}
	\end{equation}
	Furthermore,  define $\phi^*_\pi$ to be the $\alpha$-level $\pi$-Bayes test which is determined as the solution of the optimization in Section 6 of \cite{Jacobovic2020}. In particular, once $f_{\theta}(x)>0$, for every $\theta\in\Theta$ and $x\in\mathcal{X}$, then $C$ which appears in \eqref{eq:N-P necessary condition} is positive. Moreover, in such a case,  there exist $C_\pi^*\in(0,\infty)$ and $\tau^*_\pi\in[0,1]$ such that
	\begin{equation}
	\phi_\pi^*(x)=\textbf{1}_{\{L_\pi(x)>C_\pi^*\}}+\tau_\pi\textbf{1}_{\{L_\pi(x)=C_\pi^*\}}\ \ ,  \ \ \forall x\in\mathbb{R}^d
	\end{equation}
	where
	\begin{equation}
	L_\pi(x)\equiv\int_{\mathbb{R}^d\setminus\{\textbf{0}\}}\frac{f_\theta(x)}{f_{\textbf{0}}(x)}\rd\pi(\theta)\,\ \ , \ \ \forall x\in\mathbb{R}^d\,.
	\end{equation}
	
	\newpage
	The following corollary includes a sufficient condition for inadmissibility of a regular test. It stems immediately from the above-mentioned discussion with an application of Theorem \ref{thm: sufficient condition}. 
	
	\begin{corollary}\label{cor: inadmissible}
		Let $\phi\in\Phi_0(\alpha)$ and assume that there are sets $\Pi_1,\Pi_2\subseteq\Pi$ such that:
		\begin{enumerate}
			\item $\Pi=\Pi_1\cup \Pi_2$.
			
			\item 
			$\Pi_1\subseteq\left\{\pi\in\Pi\ \ ; \ \phi^*_\pi\in\Phi_0(\alpha)\right\}$.
			
			\item For every $\pi\in\Pi_2$, there exists $\phi'_\pi\in\Phi_0(\alpha)$ such that
			\begin{equation}
			\int_{\Theta_K}\beta_{\phi}(\theta)\rd\pi(\theta)<\int_{\Theta_K}\beta_{\phi'_\pi}(\theta)\rd \pi(\theta)\,.
			\end{equation}
			
		\end{enumerate} 
		If for every $\pi\in\Pi_1$, there is no trivial test for \eqref{eq: Bayes hypothesis statement} and there is no $C\in\mathbb{R}$ for which \eqref{eq:N-P necessary condition} is  satisfied,
		then $\phi\in\Phi_0(\alpha)\setminus\mathcal{A}_0(\alpha)$.
	\end{corollary}
	
	\begin{remark}
		\normalfont   
		Notice that $\Phi_0(\alpha)\subseteq\Phi(\alpha)$. Therefore, a statement that $\phi\in\Phi_0(\alpha)$ is inadmissible in $\Phi_0(\alpha)$ is more informative than a statement which asserts that it is inadmissible in $\Phi(\alpha)$.
	
	\end{remark}
	
	\begin{remark}
		\normalfont It is possible to have $\pi$ for which $\phi_\pi^*\notin\Phi_0(\alpha)$. For example, consider the special case where
		\begin{enumerate}
			\item $\Theta=\mathcal{X}=\mathbb{R}$.
			
			\item For every $\theta\in\mathbb{R}$, assume that $X\sim\mathcal{N}(\theta,1)$ under $P_\theta$.
			
			\item $\pi$ is a Dirac measure on $\{\theta_0\}$ for some $\theta_0\in\mathbb{R}\setminus\{0\}$.  
		\end{enumerate} 
		Then, $\phi_\pi^*$ is a one-sided Z-test in the direction of $\theta_0$. Clearly, this test does not belong to $\Phi_0(\alpha)$ because $\theta\mapsto\gamma_{\phi_\pi^*}(\theta)$ does not vanish in infinity. 
	\end{remark}
	It turns out that for the applications to be discussed in the upcoming sections, Corollary \ref{cor: inadmissible} may not be applied directly.
	However, these applications can be phrased as special cases of a more specific framework in which it is possible to derive a more practical result. Specifically, consider the special case where $\Theta=\mathbb{R}^p$ , $\mathcal{X}=\mathbb{R}^d\ , \ 1\leq d,p<\infty$ and $\lambda$ is the Lebesgue measure on $\mathbb{R}^d$. In addition, let $P$ be a probability measure on $\left(\Omega,\mathcal{H}\right)$ and consider $U:\Omega\rightarrow\mathbb{R}^d$ which is a $(\Omega,\mathcal{H})$-measurable function. Assume that the distribution function of $U$ (with respect to $P$) is absolutely continuous and denote the corresponding density by $f_U(\cdot)$. In particular, assume that $f_U(u)>0$ for every $u\in\mathbb{R}^d$. 
	
	For every $\theta\in\mathbb{R}^p$, let $T_\theta:\mathbb{R}^d\rightarrow\mathbb{R}^d$ be a bijective differentiable function such that for every $u\in\mathbb{R}^p$:
	\begin{description}
		\item[(T1)] $\theta\mapsto T_\theta(u)$ is continuous.
		
		\item[(T2)] $\|T_\theta(u)\|\rightarrow\infty$ as $\|\theta\|\to\infty$ where $\|\cdot\|$ is the Euclidean norm in the proper space ($\mathbb{R}^d$ or $\mathbb{R}^p$). 
	\end{description} 
	Then, for every $\theta\in\mathbb{R}^p$ define $X\equiv X(\theta)\equiv T_\theta(U)$. Correspondingly, for every $\theta\in\mathbb{R}^p$, $P_\theta$ is the push-forward probability measure which is induced by $X(\theta)$. Hence, the Jacobian theorem implies that for every $\theta\in\mathbb{R}^p$, $P_\theta$ has a density with respect to Lebesgue measure on $\mathbb{R}^d$ which is given by
	
	\begin{equation}
	f_\theta(x)\equiv f_U\left[T_\theta^{-1}\left(x\right)\right]\big|\det J_{T_\theta^{-1}}(x)\big|\ \ , \ \ \forall x\in\mathbb{R}^d
	\end{equation} 
	where $T_\theta^{-1}$ is the inverse of $T_\theta$ and $J_{T_\theta^{-1}}$ is the Jacobian matrix associated with $T^{-1}_\theta$. Importantly, by the above-mentioned assumptions, deduce that for every $\theta\in\mathbb{R}^p$, $f_\theta(\cdot)$ is positive on $\mathbb{R}^d$. This implies that for every $\theta\in\mathbb{R}^p$ the distribution of the observations under $P_\theta$ is supported on $\mathbb{R}^d$. Consequently, there is no trivial test for \eqref{eq: Bayes hypothesis statement}.

	For the statement and proof of the following theorem, let $\textbf{1}_A$ be a notation of an indicator function which is supported on a set $A$. In addition, for every $a,b\in\mathbb{R}$, let $a\wedge b\equiv\min\{a,b\}$ and denote the expectation operator with respect to $P$ by $E$.
	
	\begin{theorem}\label{thm: inadmissible}
		Let $V$ be an open set in $\mathbb{R}^d$ such that $\phi=\textbf{1}_V$ satisfies the following conditions:
		\begin{description}
			\item[(I)] $\beta_\phi(\textbf{0})=\alpha$
			
			\item[(II)] $\beta_\phi(\theta)\rightarrow1$ as $\|\theta\|\rightarrow\infty$.
		\end{description}  
		Then, $\phi\in\Phi_0(\alpha)$.
		\newline\newline
		Furthermore, assume that $\Pi_1\subseteq\Pi$  is such that for every $\pi\in\Pi_1$ the following conditions are satisfied:
		\begin{description}
			\item[(i)] $L_\pi(\cdot)$ is continuous on $\mathbb{R}^d$.
			
			\item[(ii)] There is no $C\in\mathbb{R}$ for which \eqref{eq:N-P necessary condition} is satisfied.
		\end{description} 
		Under these assumptions, if for every $\pi\in\Pi\setminus\Pi_1$, there exists $\phi'\in\Phi_0(\alpha)$ such that
		\begin{equation}
		\int_{\Theta_K}\beta_{\phi}(\theta)\rd\pi(\theta)<\int_{\Theta_K}\beta_{\phi'_\pi}(\theta)\rd \pi(\theta)\,,
		\end{equation}
		then $\phi\in\Phi_0(\alpha)\setminus\mathcal{A}_0(\alpha)$.
	\end{theorem}
	
	\begin{proof}
		Observe that
		for every $\theta\in\mathbb{R}^p$, 
		\begin{equation}
		\beta_\phi(\theta)=E\textbf{1}_V\left[T_\theta(U)\right]\,.
		\end{equation}
		Thus, since $V$ is open, due to \textbf{(T1)}, bounded convergence theorem implies that $\theta\mapsto\beta_\phi(\theta)$ is continuous. This with the assumptions \textbf{(I)} and \textbf{(II)} imply that $\phi\in\Phi_0(\alpha)$.
		
		Consider some arbitrary $\pi\in\Pi_1$
		and for every $\zeta,\eta>0$ define a test
		\begin{equation}
		\phi_{\zeta,\eta}(x)\equiv \left[\textbf{1}_{\{\|x\|<\zeta\}}\frac{\phi^*_\pi(x)\zeta}{1+\zeta}+\textbf{1}_{\{\|x\|>\eta\}}\phi(x)+\frac{1}{\eta}\right]\wedge1\ \ , \ \ \forall x\in\mathbb{R}^d\,.
		\end{equation}
		Now, observe that for every $\theta\in\mathbb{R}^p$,  
		\begin{equation}
		\beta_{\phi_{\zeta,\eta}}(\theta)=E\left\{\textbf{1}_{\{\|X(\theta)\|<\zeta\}}\frac{\phi^*_\pi\left[X(\theta)\right]\zeta}{1+\zeta}+\textbf{1}_{\{\|X(\theta)\|>\eta\}}\phi\left[X(\theta)\right]+\frac{1}{\eta}\right\}\wedge1
		\end{equation}
		and hence bounded convergence theorem implies that for every $\theta\in\mathbb{R}^p$, $(\zeta,\eta)\mapsto \beta_{\phi_{\zeta,\eta}}(\theta)$ is a continuous mapping on $\mathbb{R}^2_{++}$. Notice that for every $\eta>0$, bounded convergence theorem yields that
		
		\begin{equation}\label{eq: inequality1}
		\beta_{\phi_{\zeta,\eta}}(\textbf{0})\xrightarrow{\zeta\uparrow\infty}E\left[\phi^*_\pi(U)+\textbf{1}_{\{\|U\|>\eta\}}\phi(U)+\frac{1}{\eta}\right]\wedge1>E\phi^*_\pi(U)=\alpha\,.
		\end{equation}
		In particular, to see why the inequality in \eqref{eq: inequality1} holds, 
		recall that $\beta_{\phi^*_\pi}(\textbf{0})=\alpha<1$ which implies that
		\begin{equation}
		P\left[\phi^*_\pi(U)=1\right]<1\,.
		\end{equation}
		In addition, for every $\zeta>0$ bounded convergence yields that
		\begin{equation}
		\beta_{\phi_{\zeta,\eta}}(\textbf{0})\xrightarrow{\eta\uparrow\infty}E\textbf{1}_{\{\|U\|<\zeta\}}\frac{\phi^*_\pi(U)\zeta}{1+\zeta}<\beta_{\phi^*_\pi}(\textbf{0})=\alpha
		\end{equation}
		and hence the intermediate value theorem implies that there is a sequence $\left(\zeta_n,\eta_n\right)_{n\geq1}\subset\mathbb{R}^2_{++}$ such that:  
		\newpage
		\begin{enumerate}
			\item $(\zeta_n,\eta_n)\rightarrow(\infty,\infty)$ as $n\to\infty$.
			
			\item $\beta_{\phi_{\zeta_n,\eta_n}}(\textbf{0})=\alpha$ ,  $\forall n\geq1$. 
		\end{enumerate}
	For every $n\geq1$, denote $\phi_n\equiv\phi_{\zeta_n,\eta_n}$ and observe that there is a pointwise convergence $\phi_n\rightarrow\phi_\pi^*$ as $n\to\infty$. Therefore, bounded convergence theorem implies that
	\begin{align}
	\lim_{n\to\infty}\int_{\mathbb{R}^d\setminus\{\textbf{0}\}}\beta_{\phi_n}(\theta)\rd\pi(\theta)&=\lim_{n\to\infty}\int_{\mathbb{R}^d}\phi_n(x)\rd\mathcal{P}_\pi(x)\\&=\int_{\mathbb{R}^d}\phi_\pi^*(x)\rd\mathcal{P}_\pi(x)\nonumber\\&=\int_{\mathbb{R}^d\setminus\{\textbf{0}\}}\beta_{\phi_\pi^*}(\theta)\rd\pi(\theta)>\int_{\mathbb{R}^d\setminus\{\textbf{0}\}}\beta_{\phi}(\theta)\rd\pi(\theta)\nonumber
	\end{align}
	where the inequality is justified by \textbf{(ii)}  with the help of the fundamental lemma of Neyman and Pearson (recall that there is no trivial test under the current model assumptions). This means that there is $N\geq1$ such that $\beta_{\phi_N}(\textbf{0})=\alpha$ and for which  
	  \begin{equation}
	  \int_{\mathbb{R}^d\setminus\{\textbf{0}\}}\beta_{\phi_N}(\theta)\rd\pi(\theta)>\int_{\mathbb{R}^d\setminus\{\textbf{0}\}}\beta_{\phi}(\theta)\rd\pi(\theta)\,.
	  \end{equation}
	Hence, since $\pi$ is an arbitrary element in $\Pi_1$, then Theorem \ref{thm: sufficient condition} implies that it is left to show that $\gamma_{\phi_N}\in\mathcal{C}_0(\mathbb{R}^d)$. To this end, notice that for every $\theta\in\mathbb{R}^p$
	\begin{equation}
	E\textbf{1}_{\{\|X(\theta)\|>\eta_N\}}\phi\left[X(\theta)\right]\leq\beta_{\phi_N}(\theta)\leq1
	\end{equation}
	and it is to be shown that the left hand-side tends to one as $\|\theta\|\to\infty$.  Initially, observe that \textbf{(T2)} implies the following pointwise convergence (on $\Omega$): 
	\begin{equation}
	\big|\textbf{1}_{\{\|X(\theta)\|>\eta_N\}}\phi\left[X(\theta)\right]-\phi\left[X(\theta)\right]\big|\rightarrow 0\ \ \text{as}\ \ \|\theta\|\to\infty\,.
	\end{equation}
	Consequently, 
	\begin{equation}
	\big|E\textbf{1}_{\{\|X(\theta_N)\|>\eta\}}\phi\left[X(\theta)\right]-\beta_\phi(\theta)\big|\leq E\big|\textbf{1}_{\{\|X(\theta)\|>\eta_N\}}\phi\left[X(\theta)\right]-\phi\left[X(\theta)\right]\big|
	\end{equation}
	and the right hand-side tends to zero as $\|\theta\|\to\infty$ by bounded convergence theorem. Hence, due to \textbf{(II)}, deduce that
	\begin{equation}
	E\textbf{1}_{\{\|X(\theta)\|>\eta\}}\phi\left[X(\theta)\right]\rightarrow1\ \ \text{as}\ \ \|\theta\|\to\infty\,.
	\end{equation}
	Now, \textbf{(i)} implies that
	\begin{equation}
	R\equiv\left\{x\in\mathbb{R}^d;L_\pi(x)>C_\pi^*\right\}\ , \ Q\equiv\left\{x\in\mathbb{R}^d; L_\pi(x)<C_\pi^*\right\}
	\end{equation}  
	are open sets. In addition, observe that for every $x\in\mathbb{R}^d$, $\phi_N(x)$ equals to
	\begin{equation}\label{eq: indicator representation}
	\left\{\zeta_N\frac{\textbf{1}_{R\cap B}(x)+\tau^*_\pi\left[\textbf{1}_B(x)-\textbf{1}_{R\cap B}(x)-\textbf{1}_{Q\cap B}(x)\right]}{\zeta_N+1}+\textbf{1}_{V\cap D}(x)+\frac{1}{\eta}\right\}\wedge1
	\end{equation}
	where $B\equiv\{x\in\mathbb{R}^d;\|x\|<\zeta_N\}$ and $D\equiv\{x\in\mathbb{R}^d;\|x\|>\eta_N\}$. Since $V$ is an open set, then each of the indicators which appear in \eqref{eq: indicator representation} is supported on an open set. This means that the same arguments which were introduced in order to prove continuity of $\theta\mapsto\beta_\phi(\theta)$ imply that $\theta\mapsto\beta_{\phi_N}$ is continuous.
	\end{proof}

	\section{One-dimensional general case} \label{sec: univariate} 
	Let $\theta\in\mathbb{R}$ and consider two independent random variables $U_1$ and $U_2$ which are not necessarily identically distributed. In particular, assume that for every $i=1,2$, the distribution of $U_i$ is absolutely continuous with a continuous density function $f_{U_i}:\mathbb{R}\rightarrow(0,M)$ for some $0<M<\infty$.  In addition, define 
	\begin{equation}
	X(\theta)\equiv X=\begin{bmatrix}
	X_1 \\ X_2
	\end{bmatrix}=\begin{bmatrix}
	U_1 +\theta \\ U_2+\theta
	\end{bmatrix}
	\end{equation}
	and let $P_\theta$ be the push-forward probability measure which is induced by $X(\theta)$. 
	
	Then, a statistician who does not know the value of $\theta$ (but knows the distribution of $U$) observes $X$ and wants to test
	\begin{equation}\label{eq: one-dimensionalstatement}
	H:\theta=0 \ \ , \ \ K:\theta\neq0
	\end{equation}
	with a pre-defined significance level $\alpha\in(0,1)$.
	For this purpose, he considers a test 
	\begin{equation}\label{eq: splitting test}
	\phi(x)\equiv\textbf{\textbf{1}}_{(a,\infty)\times (b_1,\infty)\cup (-\infty,a)\times (-\infty,b_2)}\left(x_1,x_2\right)\ \ , \ \ \forall x=(x_1,x_2)\in\mathbb{R}^2
	\end{equation}
	for some $a\in\mathbb{R}$ 
	where $b_1$ and $b_2$ are the $(1-\alpha)$'th and $\alpha$'th quantiles of the distribution of $U_2$, \textit{i.e.,} 
	\begin{align}
	&b_1\equiv\inf\left\{t\in\mathbb{R};P(U_2\leq t)\geq1-\alpha\right\}\,,\\&b_2\equiv\inf\left\{t\in\mathbb{R};P(U_2\leq t)\geq\alpha\right\}\,.\nonumber
	\end{align}
	Note that $\phi$ is a test which is based on data-splitting in the sense that $X_1$ is used in order to determine the direction of a one-sided test to be performed through the statistic $X_2$. In particular, regardless of the value of $X_1$, observe that under $H$, the test to be performed through $X_2$ leads to a rejection with probability $\alpha$ which means that $\phi\in\Phi(\alpha)$.
	
	\begin{theorem}\label{prop: one-dimensionalcase}$\text{   }$
	 	\begin{enumerate}
	 		\item  $\phi\in\Phi_0(\alpha)$.
	 		
	 		\item If 
	 		\begin{equation}\label{eq: condition 1}
	 		\exists\lim_{x\to\infty}\frac{f_{U_2}(x-\theta)}{f_{U_2}(x)}=\infty\ \ , \ \ \exists\lim_{x\to-\infty}\frac{f_{U_2}(x-\theta)}{f_{U_2}(x)}<\infty\ \ , \ \ \forall\theta\in(0,\infty)
	 		\end{equation}
	 		and
	 		\begin{equation} \label{eq: condition 2}
	 		\exists\lim_{x\to\infty}\frac{f_{U_2}(x-\theta)}{f_{U_2}(x)}<\infty\ \ , \ \ \exists\lim_{x\to-\infty}\frac{f_{U_2}(x-\theta)}{f_{U_2}(x)}=\infty\ \ , \ \ \forall\theta\in(-\infty,0)\,.
	 		\end{equation}
	 		Then,  $\phi\in\Phi_0(\alpha)\setminus\mathcal{A}_0(\alpha)$.
	 	\end{enumerate}
	 	
	 \end{theorem}
 	
 	\begin{proof}
 		It is easy to verify that $\phi$ satisfies the assumptions of the first part of Theorem \ref{thm: inadmissible} which makes the first assertion follows. 
 		
 		In order to prove the second assertion, consider a probability measure $\pi\in\Pi$ which is concentrated on $(0,\infty)$ and define a test
 		\begin{equation}
 		\phi^+_\zeta(x_1,x_2)\equiv\textbf{\textbf{1}}_{(\zeta,\infty)\times (b_1,\infty)\cup (-\infty,\zeta)\times (-\infty,b_2)}\left(x_1,x_2\right)\ \ , \ \ \forall(x_1,x_2)\in\mathbb{R}^2
 		\end{equation}
 		which is parametrized by $\zeta\in(-\infty,a)$. Notice that for every $\zeta\in(-\infty,a)$, $\phi_\zeta^+\in\Phi_0(\alpha)$ just like $\phi\in\Phi_0(\alpha)$. 
 		
 		Since $f_{U_2}(\cdot)$ is supported on $\mathbb{R}$, then the definitions of $b_1$ and $b_2$ imply that 
 		\begin{equation}
 		P_{\theta}\left(X_2< b_2\right)<\alpha<P_{\theta}\left(X_2> b_1\right)\ \ , \ \ \forall\theta\in(0,\infty)\,.
 		\end{equation}
 		Therefore, since for every $\theta\in(0,\infty)$, $X_1$ and $X_2$ are independent and $P_\theta(X_1)=0$, then 
 		\begin{align}\label{eq: inequality}
 		P_\theta\left(X_2> b_1\right)&=P_{\theta}\left(X_1> a\right)P_{\theta}\left(X_2> b_1\right)\nonumber+P_{\theta}\left(X_1<a \right)P_{\theta}\left(X_2> b_1\right)\nonumber\\&> P_{\theta}\left(X_1>a\right)P_{\theta}\left(X_2> b_1\right)\nonumber+P_{\theta}\left(X_1<a\right)P_{\theta}\left(X_2< b_2\right)\nonumber\\&=\beta_\phi(\theta)\ \ , \ \ \forall\theta\in(0,\infty)\,.
 		\end{align}
 		Note that for every $\theta\in(0,\infty)$, bounded convergence implies that
 		\begin{align}
 		\beta_{\phi_\zeta^+}(\theta)&=E\textbf{\textbf{1}}_{(\zeta,\infty)\times (b_1,\infty)\cup (-\infty,\zeta)\times (-\infty,b_2)}\left[X_1(\theta),X_2(\theta)\right]\\&\rightarrow P_\theta\left(X_2> b_1\right)\ \ \text{as} \ \ \zeta\to-\infty\,.\nonumber
 		\end{align}
 		Therefore, an additional application of bounded convergence theorem yields that
 		\begin{align}
 		\lim_{\zeta\to-\infty}\int_{(0,\infty)}\beta_{\phi_\zeta^+}(\theta)\rd\pi(\theta)\nonumber&=\int_{(0,\infty)}P_\theta\left(X_2> b_1\right)\rd\pi(\theta)\\&>\int_{(0,\infty)}\beta_\phi(\theta)\rd\pi(\theta)
 		\end{align}
 		where the inequality is justified by \eqref{eq: inequality}. This means that for any $\pi\in\Pi$ which is concentrated on $(0,\infty)$, $\phi$ is not an $\alpha$-level $\pi$-Bayes test in $\Phi_0(\alpha)$. Similarly, an analogue result can be made for any $\pi\in\Pi$ which is concentrated on $(-\infty,0)$.
 		
 		Thus, it is left to consider $\pi\in\Pi$ which is a probability measure on $\mathbb{R}\setminus\{0\}$ such that 
 		\begin{equation}
 		\pi\left((-\infty,0)\right)\wedge\pi\left((0,\infty)\right)>0\,.
 		\end{equation}
 		Primarily, for every $\theta\in\mathbb{R}$, the joint likelihood of the data is given by
 		\begin{equation}\label{eq: joint likelihood}
 		l_\theta(x_1,x_2)\equiv f_{U_1}(x_1-\theta)f_{U_2}(x_2-\theta)\ \ , \ \ \forall (x_1,x_2)\in\mathbb{R}^2\,.
 		\end{equation}
 		In particular, it is given that $f_{U_1}(\cdot)$ and $f_{U_2}(\cdot)$ are positive continuous functions on $\mathbb{R}$. Therefore, deduce that for every $\theta\in\mathbb{R}\setminus\{0\}$, the likelihood ratio
 		\begin{equation}
 		(x_1,x_2)\mapsto\frac{l_\theta(x_1,x_2)}{l_0(x_1,x_2)}
 		\end{equation}
 		is continuous on $\mathbb{R}^2$. Consequently, since for every $(x_1,x_2)\in\mathbb{R}^2$, $\theta\mapsto l_\theta(x_1,x_2)$ is bounded, then bounded convergence theorem leads to the conclusion that
 		\begin{equation}
 		L_\pi(x_1,x_2)=\int_{\mathbb{R}\setminus\{0\}} \frac{l_\theta(x_1,x_2)}{l_0(x_1,x_2)}\rd\pi(\theta)\ \ , \ \ \forall(x_1,x_2)\in\mathbb{R}^2
 		\end{equation} 
 		is continuous. 
 		
 		Now, observe that for every $(x_1,x_2)\in\mathbb{R}^2$ 
 		\begin{equation}\label{eq: likelihood decomposition}
 		L_\pi(x_1,x_2)=\int_{(-\infty,0)} \frac{l_\theta(x_1,x_2)}{l_0(x_1,x_2)}\rd\pi(\theta)+\int_{(0,\infty)} \frac{l_\theta(x_1,x_2)}{l_0(x_1,x_2)}\rd\pi(\theta)\,.
 		\end{equation}
 		In addition, fix some $x_1\neq0$ and notice that \eqref{eq: condition 1} and \eqref{eq: condition 2} with bounded convergence theorem imply that
 		\begin{equation}
 		\exists\lim_{x_2\to\pm\infty}L_\pi(x_1,x_2)=\infty\,.
 		\end{equation}
 		This means that for every $C\in(0,\infty)$, there exist $-\infty<w_1<w_2<\infty$ such that
 		\begin{equation}
 		(-\infty,w_1)\cup(w_2,\infty)\subseteq\{x_2\in\mathbb{R}\ ;\ L_\pi(x_1,x_2)>C\}\,.
 		\end{equation}   
 		 On the other hand,
 		\begin{equation}
 		\left\{x_2\in\mathbb{R}\ ;\ \phi(x_1,x_2)=1\right\}=\begin{cases} 
 		(b_1,\infty) & x_1>0 \\
 		(-\infty,b_2) & x_1<0
 		\end{cases}
 		\end{equation} 
 		and hence the result follows by an application of Theorem \ref{thm: inadmissible}.
 	\end{proof}

	\begin{remark}\label{remark: one-dimensionalGaussian}
		\normalfont When $U_2$ has a distribution which is symmetric with respect to the origin, \textit{i.e.,} $f_{U_2}(x)=f_{U_2}(-x)$ almost everywhere on $\mathbb{R}$, the conditions \eqref{eq: condition 1} and \eqref{eq: condition 2} may be unified into the condition
		\begin{equation}\label{eq: symmetric condition}
		\exists\lim_{x\to\infty}\frac{f_{U_2}(x-\theta)}{f_{U_2}(x)}=\infty\ \ , \ \ \forall\theta\in(0,\infty)\,.
		\end{equation}
		Especially, in the Gaussian case, \textit{i.e.,} when 
		\begin{equation}
		f_{U_2}(x)=\frac{1}{\sqrt{2\pi}}e^{-\frac{x^2}{2}}\ \ , \ \ \forall x\in\mathbb{R}\,,
		\end{equation}
		simple algebra implies that
		\begin{equation}
		\frac{f_{U_2}(x-\theta)}{f_{U_2}(x)}\propto e^{\theta x}\ \ , \ \ \forall \theta\in\mathbb{R}\setminus\{0\}\,.
		\end{equation}   
		Hence, Theorem \ref{prop: one-dimensionalcase} yields that Moran's test is inadmissible in the Gaussian case.  
		 
	\end{remark} 
	
	\begin{remark}
		\normalfont It makes sense to think about a model in which  $\xi_1,\ldots,\xi_n$ is an iid sequence of random variables. Then, a statistician who does not know the value of $\theta\in\mathbb{R}$ observes  $\rho_i=\xi_i+\theta$ ,  $i=1,2,\ldots,n$. His intention is to test \eqref{eq: one-dimensionalstatement} with a procedure which is based on a single split of the data. What is the relation between this setup and the model which was described so far in this section? In this setup $X_1$ and $X_2$ may be viewed as statistics which are calculated respectively from the first and second sub-samples. Assume that the first sub-sample is $\left(\rho_1,\ldots,\rho_m\right)$ and the second sub-sample is $\left(\rho_{m+1},\ldots,\rho_n\right)$ for some $1\leq m<n$. Then, the requirement is that $U_1\equiv X_1-\theta$ and $U_2\equiv X_2-\theta$ satisfy the following conditions:
		\newpage
		\begin{enumerate}
			\item $U_1$ is  determined uniquely by $\left(\xi_1,\ldots,\xi_m\right)$ and $U_2$ is  determined uniquely by $\left(\xi_{m+1},\ldots,\xi_n\right)$. 
			
			\item The distribution of $(U_1,U_2)$ is free of $\theta$.   
		\end{enumerate}
		 Some examples of statistics which satisfy these conditions are \textit{e.g.,} sample means, sample quantiles and sample extreme values. Notably, in Section 2 of \cite{Moran1973}, Moran regards the case in which $X_1$ and $X_2$ are the sample means of the two sub-samples.
	\end{remark}
	
	\section{Multi-dimensional Gaussian case}\label{sec: multivariate}
	Let $d\geq1$ and assume that $U_1,U_2,\ldots,U_n$ is an iid sequence of standard $d$-dimensional Gaussians. In addition, for every $1\leq i\leq n$, define $X_i\equiv X_i(\theta)\equiv U_i+\theta$. A statistician observes $X_1,X_2,\ldots,X_n$ and his purpose is to test
	
	\begin{equation}
	H:\theta=\textbf{0} \ \ , \ \ K:\theta\in\mathbb{R}^d\setminus\{\textbf{0}\}
	\end{equation}
	where $\textbf{0}$ is the zero-vector in $\mathbb{R}^d$. In this model, Moran's test $\phi$ is as follows: The first step is to compute the mean of the first sub-sample, that is $\bar{X}_1\equiv\frac{1}{m}\sum_{i=1}^mX_i$ for some $1\leq m<n$. Then, given the computation result, use the second sub-sample $X_{m+1},\ldots,X_n$ in order to test a simple hypothesis in the direction of $\bar{X}_1$. That is, consider
	\begin{equation}
	H:\theta=\textbf{0} \ \ , \ \ K(\bar{X}_1):\theta=\frac{\bar{X}_1}{\|\bar{X}_1\|}
	\end{equation}
	with a pre-defined significant level $\alpha\in\left(0,1\right)$ where $\|\cdot\|$ is the Euclidean norm in $\mathbb{R}^d$. The general idea for this test was presented by Moran in  \cite{Moran1973} and recently, DiCiccio analysed the power of this test (see, Section 2.1.2 in \cite{DiCiccio2018}).
	
	In practice, straightforward calculation of the likelihood ratio (note that $x_1$ is considered as a constant) implies that the  rejection zone of $\phi$ equals to
	
	\begin{equation}
	R_\phi\equiv\left\{x_1,x_2,\ldots,x_n\in\mathbb{R}^d\ \text{s.t.}\ \bar{x}_1\neq\textbf{0}\ ; \ \frac{\bar{x}_1\cdot\bar{x}_2}{\|\bar{x}_1\|}>D\right\}
	\end{equation}
	where $\cdot$ denotes dot product, $\bar{x}_1\equiv\frac{1}{m}\sum_{i=1}^mx_i$, $\bar{x}_2\equiv\frac{1}{n-m}\sum_{i=m+1}^nx_i$ and $D>0$ is a constant which is determined uniquely by the vector $(\alpha,n,m)$.
	\newpage
	\begin{theorem}\label{prop: multivariate Gaussian}
		$\phi\in\Phi_0(\alpha)\setminus\mathcal{A}_0(\alpha)$. 
	\end{theorem}
	
	\begin{proof}
		For simplicity and w.l.o.g. consider the case where $n=2$ and $m=1$. In addition, as mentioned in Remark \ref{remark: one-dimensionalGaussian}, Theorem \ref{prop: one-dimensionalcase} implies the result for the special case where $d=1$ and hence consider the case where $2\leq d<\infty$.
		
		By construction, $\phi\in\Phi(\alpha)$. In addition,  notice that for every $\theta\neq\textbf{0}$
		\begin{align}
		\frac{X_1(\theta)\cdot X_2(\theta)}{\|X_1(\theta)\|}&=\frac{\left(U_1+\theta\right)\cdot\left(U_2+\theta\right)}{\|U_1+\theta\|}\\&= \frac{U_1\cdot U_2+\theta\cdot(U_1+U_2)+\|\theta\|^2}{\|U_1+\theta\|}\nonumber\\&\geq\frac{U_1\cdot U_2-\|\theta\|\|U_1+U_2\|+\|\theta\|^2}{\|U_1+\theta\|}\nonumber
		\end{align}
		and observe that
		\begin{equation}
		0<\frac{\|U_1+\theta\|}{\|\theta\|}\leq\frac{\|U_1\|}{\|\theta\|}+1\,.
		\end{equation}
		Therefore, deduce that
		\begin{equation}
		\frac{X_1(\theta)\cdot X_2(\theta)}{\|X_1(\theta)\|}\xrightarrow{w.p.1}\infty\ \ \text{as} \ \ \|\theta\|\to\infty
		\end{equation}
		and hence $\beta_\phi(\theta)\rightarrow1$ as $\|\theta\|\to\infty$.
		Also, observe that $(x_1,x_2)\mapsto\frac{x_1\cdot x_2}{\|x_1\|}$ is a continuous function on $\{x_1,x_2\in\mathbb{R}^d;x_1\neq\textbf{0}\}$ which means that $R_\phi$ is an open set. Therefore, by the first part of Theorem \ref{thm: inadmissible}, deduce that $\phi\in\Phi_0(\alpha)$.
		
		Now, consider some $\pi\in\Pi$ and observe that the same arguments which appear in the proof of Theorem \ref{prop: one-dimensionalcase} may be used here in order to show that $(x_1,x_2)\mapsto L_\pi(x_1,x_2)$ is continuous.	In addition, standard algebra implies that for every $x_1,x_2\in\mathbb{R}^d$ 
		
		\begin{align}
		L_\pi(x_1,x_2)&=\int_{\mathbb{R}^d\setminus\{\textbf{0}\}}\exp\left\{-\frac{1}{2}\sum_{i=1,2}\left[\|\theta\|^2-2\theta\cdot x_i\right]\right\}\rd\pi(\theta)\,.\nonumber
		\end{align}  
		Since expectation is an operator which preserves convexity, deduce that $(x_1,x_2)\mapsto L_\pi(x_1,x_2)$ is convex on $\mathbb{R}^{2d}$. Consequently, for every $C\in(0,\infty)$ the set
		\begin{equation}
		\left\{\left(x_1,x_2\right)\in\mathbb{R}^{2d}\ ;\ L_\pi(x_1,x_2)\leq C\right\}
		\end{equation}
		is convex. 
		
		Now, assume by contradiction that $\phi$ is an $\alpha$-level $\pi$-Bayes test. Therefore, up to a null set, the acceptance zone of $\phi$ is convex. More precisely, this implies that the Lebesgue measure of the set of points $(u,v)\in\mathbb{R}^{2d}$ for which there exist $(u_1,v_1),(u_2,v_2)\in\mathbb{R}^{2d}$ such that
		\begin{enumerate}
			\item $(u,v)=\frac{(u_1,v_1)+(u_2,v_2)}{2}$ ,
			
			\item $\frac{u\cdot v}{\|u\|}> D$ ,
			
			\item  $\frac{u_i\cdot v_i}{\|u_i\|}< D\ \ , \ \ \forall i=1,2$ ,
		\end{enumerate}
		is zero. In order to obtain a contradiction, denote
		\begin{equation}
		u_1\equiv\frac{\textbf{e}_1}{\sqrt{2}}+\frac{\textbf{e}_2}{\sqrt{2}}\ \ , \ \ u_2\equiv\frac{\textbf{e}_1}{\sqrt{2}}-\frac{\textbf{e}_2}{\sqrt{2}}
		\end{equation} 
		and $v_1\equiv v_2\equiv\delta D\textbf{e}_1$ for some $1<\delta<\sqrt{2}$ where $\textbf{e}_i$ is the $i$'th ($i=1,2$) element in the standard basis of $\mathbb{R}^d$. In particular observe that for every $i=1,2$, 
		
		\begin{equation}
		\frac{u_i\cdot v_i}{\|u_i\|}=\frac{\delta D}{\sqrt{2}}<D\,.
		\end{equation}
		In addition, define 
		\begin{equation}
		u\equiv\frac{u_1+u_2}{2}=\frac{\textbf{e}_1}{\sqrt{2}}\ \ , \ \ v\equiv\frac{v_1+v_2}{2}=\delta D\textbf{e}_1
		\end{equation}
		and notice that 
		\begin{equation}
		\frac{u\cdot v}{\|u\|}=\frac{\delta D/\sqrt{2}}{1/\sqrt{2}}=\delta D>D\,.
		\end{equation}
		Thus, a continuity argument yields that there exists $\delta>0$ such that 
		\begin{equation}
		\frac{\left(\frac{\tilde{u}_1+\tilde{u}_2}{2}\right)\cdot \left(\frac{\tilde{v}_1+\tilde{v}_2}{2}\right)}{\|\frac{\tilde{u}_1+\tilde{u}_2}{2}\|}>D\ \ , \ \ \frac{\tilde{u}_1\cdot\tilde{v}_1}{\|\tilde{u}_1\|}< D\ \ ,\ \  \frac{\tilde{u}_2\cdot\tilde{v}_2}{\|\tilde{u}_2\|}< D 
		\end{equation}
	for every	
	$(\tilde{u}_i,\tilde{v}_i)\in B_\delta\left((u_i,v_i)\right)\ , \ i=1,2$ where $B_\delta(y)$ refers to an Euclidean ball with radius $\delta>0$ around $y\in\mathbb{R}^{2d}$. Thus, the result follows by Theorem \ref{thm: inadmissible}. 
	\end{proof}
	
	\begin{remark}\normalfont
		The initial effort was to prove a multi-dimensional version of Theorem \ref{prop: one-dimensionalcase}, \textit{i.e.}, when $U_i$ ($1\leq i\leq n$) has a general multi-dimensional distribution. Observe that the multi-dimensional setup implies existence of a continuum of directions. Therefore, a generalization of the proof which appears in Section \ref{sec: univariate} is not straightforward. 
	\end{remark}
	
	\section{Discussion}\label{sec:conclusion}
	From the perspective of classical decision theory, 
	an inadmissible test should not be used because there is another test which is better. Accordingly, the inadmissibility results which appear in this work are not encouraging applied statisticians to apply Moran's test for real data. The purpose of this short section is to discuss the implications of the current results more deeply along with some suggestions for further research.
	
	Primarily, knowing that Moran's test is inadmissible should be considered as an initial step toward the pursuit for a better test. For example, in the one-dimensional Gaussian case, a combination of Theorem \ref{prop: one-dimensionalcase} and the theory regarding unbiased tests  implies that a two-sided 	Z-test is better than Moran's test. However, it is not clear how to find a better test in the general case, \textit{e.g.,} when Condition \ref{cond: inadmissibility} is satisfied. This practical question remains open.
	
	Furthermore, it is possible that in certain setups, Moran's test performs quite well (at least for alternatives which are distant from the null). In such cases, there is no strong incentive for practitioners to look for a better test. In Section 2.1.2 of \cite{DiCiccio2018}, there is an effort to compare the power of Moran's test with the chi-square test in the multi-dimensional Gaussian setup. It might be good to keep on in this direction by assessing the performance of Moran's test when the data  does not have a Gaussian distribution.      
	
	Finally, it is also reasonable to consider a statistician who is willing to pay in terms of power for testing the `correct' hypothesis. It is interesting to see how to phrase a formal model which is consistent with the preferences of such a statistician. Then, the challenge will be to figure out whether Moran's test is admissible in this new framework.  Possibly, an inspiration for such a model might come from some model-selection frameworks in which the statistician is willing to give up some part of the data in favour of conducting a statistical inference on a better model. Another branch of literature which might be related regards constrained statistical inference and an analysis of type III error (for more information, see \textit{e.g.,} \cite{Silvapulle2005}).  
	
\begin{equation*}
\end{equation*}
	\textbf{Acknowledgement:} The author would like to thank Ori Davidov for interesting discussions which help in finding the topic for this work.

\end{document}